\documentclass{amsart}
\usepackage[latin1]{inputenc}
\usepackage{amsmath,amssymb,graphicx,setspace,verbatim}
\usepackage{color}
\usepackage{comment}
\usepackage{appendix}
\theoremstyle{comment}

\newtheorem*{mcomment}{\color{cyan}{Comment}}

\newcommand{\Sym}{\mathrm{S}} 

\input xy
\xyoption{all}

\newtheorem{theorem}{Theorem}[section]

\newtheorem{lemma}[theorem]{Lemma}
\newtheorem{corollary}[theorem]{Corollary}
\newtheorem{proposition}[theorem]{Proposition}
\theoremstyle{definition}

\begin{document}

\title{C-groups of high rank for the symmetric groups}

\author{Maria Elisa Fernandes}
\address{
Maria Elisa Fernandes, Center for Research and Development in Mathematics and Applications, Department of Mathematics, University of Aveiro, Portugal
}
\email{maria.elisa@ua.pt}

\author{Dimitri Leemans}
\address{Dimitri Leemans, Department of Mathematics, University of Auckland, Private Bag 92019, Auckland 1142, New Zealand
}
\email{d.leemans@auckland.ac.nz}

\begin{abstract}
We classify C-groups of ranks $n-1$ and $n-2$ for the symmetric group $S_n$.
We also show that all these C-groups correspond to hypertopes, that is, thin, residually connected flag-transitive geometries. Therefore we generalise some similar results obtained in the framework of string C-groups that are in one-to-one correspondence with abstract regular polytopes.
\end{abstract}
\maketitle

\noindent \textbf{2000 Math Subj. Class:} 52B11, 20D06.

\section{Introduction}
In 1896, Eliakim H.~Moore~\cite{Moore1896} gave, for the symmetric group $S_n$, a generating set consisting of the $n-1$ transpositions $(i,i+1)$, $(i = 1, \ldots, n-1)$.
This set is closely linked to the $(n-1)$-simplex and, as it was shown in~\cite{FL1}, it gives the unique abstract regular polytope of rank $n-1$ for $S_n$ when $n\geq 5$.
In~\cite{FL1}, the authors also proved that, up to isomorphism and duality, there is only one abstract regular polytope of rank $n-2$ for  $S_n$ when $n\geq 7$. 

In 1997, Francis Buekenhout, Philippe Cara and Michel Dehon~\cite{BCD97} introduced inductively minimal geometries. 
It turns out that inductively minimal geometries of rank $n-1$ are all thin with automorphism group the symmetric group $S_n$.
The $(n-1)$-simplex is one of them.

There exists a one-to-one correspondence between the set of non-isomorphic inductively minimal geometries of rank $n-1$ and the set of
    non-isomorphic trees with $n$ vertices (see \cite{CLP}). Moreover, the diagrams of these geometries are the line graphs of the corresponding trees. 

In 2002, Julius Whiston~\cite{Whis00} showed that the size of an independent generating set in the symmetric group $S_n$ is at most $n-1$. In \cite{CaCa:ta}, Peter Cameron and Cara determined all sets meeting this bound and gave a bijection between independent generating sets of size $n-1$ (up to conjugation and inversion of some generators) and residually weakly primitive  coset geometries of rank $n -1$ for the symmetric group $S_n$. 
Particularly, the geometries arising from a group generated by a set of $n-1$ transpositions in $S_n$ corresponding to the edges of a tree, are 
precisely, the inductively minimal geometries of rank $n-1$. 

There is a well known correspondence between polytopes and geometries. Indeed polytopes are thin residually connected geometries with a linear diagram.    
    In the present paper we consider geometries whose diagram need not be linear  and we generalise Theorems 1 and 2  of \cite{FL1} giving a classification of regular hypertopes (also known as thin regular residually connected geometries) of ranks $n-1$ and $n-2$ for $S_n$. The automorphism groups of hypertopes are C-groups (see \cite{FLW1}), thus we characterise the geometries in terms of finite quotients of certain Coxeter groups.

C-groups of rank $n-1$ for $S_n$ are in one-to-one correspondence with 
regular hypertopes of rank $n-1$, that are precisely the inductively minimal geometries of rank $n-1$, thus the classification follows from \cite{CaCa:ta}. 
\begin{theorem}\label{CaCa}\cite{CaCa:ta}
For $n\geq 7$, $G$ a permutation group of degree $n$ and  $\{\rho_0, \ldots, \rho_{n-2}\}$ a set of involutions of $G$, $\Gamma = (G, \{\rho_0, \ldots, \rho_{n-2}\})$ is a C-group of rank $n-1$ if and only if the permutation representation graph of $\Gamma$ is a tree $\mathcal{T}$ with $n$ vertices.
Moreover, the $\rho_i$'s are transpositions, $G\cong S_n$ and the Coxeter diagram of $\Gamma$ is the line graph of $\mathcal{T}$.
Finally, every such C-group gives a regular hypertope of rank $n-1$ for $S_n$.

\end{theorem}

The main result of this paper is the following theorem that gives a classification of C-groups of rank $n-2$ for $S_n$.

\begin{theorem}\label{main}
Let  $n\geq 9$. Then $\Gamma = (S_n, \{\rho_0, \ldots, \rho_{n-3}\})$ is a C-group of rank $n-2$ if and only if 
its permutation representation graph belongs to one the following three families, up to a renumbering of the generators, where  $\rho_2, \ldots, \rho_{n-3}$ are transpositions corresponding to the edges of a tree with $n-3$ vertices and the two remaining involutions, $\rho_0$ and $\rho_1$, are either transpositions or 2-transpositions (with at least one of them being a 2-transposition).
$$(A)
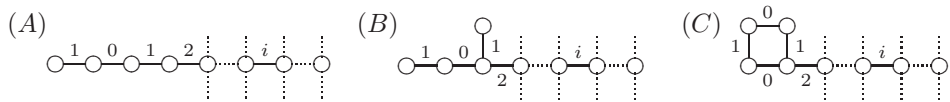
\xymatrix@-1.3pc{&&&&&&&\\*+[o][F]{}   \ar@{-}[r]^1& *+[o][F]{}   \ar@{-}[r]^0 & *+[o][F]{} \ar@{-}[r]^1& *+[o][F]{}\ar@{-}[r] ^2& *+[o][F]{}\ar@{.}[r]\ar@{.}[d]\ar@{.}[u]& *+[o][F]{}\ar@{-}[r]^i\ar@{.}[d]\ar@{.}[u]& *+[o][F]{}\ar@{.}[r] \ar@{.}[d]\ar@{.}[u]&*+[o][F]{}\ar@{.}[d]\ar@{.}[u]\\
&&&&&&&} \quad  (B) \xymatrix@-1.3pc{&&*+[o][F]{}  &&&&&\\*+[o][F]{}   \ar@{-}[r]^1& *+[o][F]{}   \ar@{-}[r]^0 & *+[o][F]{} \ar@{-}[u]_1\ar@{-}[r]_2 & *+[o][F]{}\ar@{.}[r]\ar@{.}[d]\ar@{.}[u]& *+[o][F]{}\ar@{-}[r]^i\ar@{.}[d]\ar@{.}[u]& *+[o][F]{}\ar@{.}[r] \ar@{.}[d]\ar@{.}[u]&*+[o][F]{}\ar@{.}[d]\ar@{.}[u]\\
&&&&&&&} (C) \xymatrix@-1.3pc{ *+[o][F]{}   \ar@{-}[r]^0\ar@{-}[d]_1  & *+[o][F]{} \ar@{-}[d]^1&&&&&& \\
*+[o][F]{}   \ar@{-}[r]_0  & *+[o][F]{} \ar@{-}[r]_2 & *+[o][F]{}\ar@{.}[r]\ar@{.}[d]\ar@{.}[u]& *+[o][F]{}\ar@{-}[r]^i\ar@{.}[d]\ar@{.}[u]& *+[o][F]{}\ar@{.}[r] \ar@{.}[d]\ar@{.}[u]&*+[o][F]{}\ar@{.}[d]\ar@{.}[u]\\
&&&&&&&} $$
Moreover, every such C-group gives a regular hypertope of rank $n-2$ for $S_n$.
\end{theorem}

In order to prove this theorem, we use group theory and especially knowledge of primitive and transitive imprimitive groups, but also elementary graph theory while investigating permutation representation graphs.

In Section~\ref{back} we give the background needed for the understanding of this paper. In Section~\ref{n-1} we focus on the classification of C-groups of rank $n-1$ for $S_n$, determining a presentation for these groups.
In Section~\ref{transitive} we show that all maximal parabolic subgroups of a C-group of rank $n-2$ for $S_n$ must be intransitive when $n\geq 9$.  In Section~\ref{intransitive} we deal with the case where all maximal parabolic subgroups are intransitive and determine the possible shapes of the permutation representation graphs of these C-groups. Finally, in Section~\ref{proof}, we give the proof of Theorem~\ref{main} and we obtain new presentations for the groups $S_n$ from the permutation representation graphs appearing in Theorem~\ref{main}.

We follow notation of the Atlas~\cite{Con85} for groups.
\section{background}\label{back}

\subsection{C-groups}

Let $G$ be a group generated by $r$ involutions $\rho_0,\ldots,\rho_{r-1}$. 
Then $\Gamma:=(G, \{\rho_0,\ldots,\rho_{r-1}\})$ is called a \emph{C-group of rank $r$} if $\Gamma$ satisfies the \emph{intersection property}~(\ref{ip}) with respect to its 
generators; that is, 
\begin{equation}\label{ip}
\forall J, K \subseteq \{0,\ldots,r-1\}, \langle \rho_j \mid j \in J\rangle \cap \langle \rho_k \mid k \in K\rangle = \langle \rho_j \mid j \in J\cap K\rangle
\end{equation}
Here, ``C'' stands for ``Coxeter'', as a Coxeter group is a C-group. 
Let $p_{i,j}$ be the order of $\rho_i\rho_j$. The \emph{Coxeter diagram} of $\Gamma$ is a graph whose vertices are the generators $\rho_0,\ldots,\rho_{r-1}$ and with an edge $\{\rho_i,\rho_j\}$ whenever  $p_{i,j}>2$. An edge $\{\rho_i,\rho_j\}$ of the Coxeter graph has a label $p_{i,j}$ when $p_{i,j}>3$ and it has no label when $p_{i,j}=3$.
The subgroups $\Gamma_i = \langle \rho_j : j \in \{0, \ldots, r-1\} \setminus \{i\}\rangle$ with $i=0, \ldots, r-1$ are called the \emph{maximal parabolic subgroups of $\Gamma$}.
Finally, the \emph{C-rank} of a group $G$ is the maximal size $r$ of a set of generators $\{\rho_0,\ldots,\rho_{r-1}\}$ of $G$ that satisfies~(\ref{ip}).
\subsection{Regular hypertopes}

As in~\cite{BuekCohen}, an {\it incidence system} $\Gamma := (X, *, t, I)$ is a 4-tuple such that
\begin{itemize}
\item $X$ is a set whose elements are called the {\it elements} of $\Gamma$;
\item $I$ is a set whose elements are called the {\it types} of $\Gamma$;
\item $t:X\rightarrow I$ is a {\it type function}, associating to each element $x\in X$ of $\Gamma$ a type $t(x)\in I$;
\item $*$ is a binary relation on $X$ called {\em incidence}, that is reflexive, symmetric and such that for all $x,y\in X$, if $x*y$ and $t(x) = t(y)$ then $x=y$.
\end{itemize}
The {\it incidence graph} of $\Gamma$ is the graph whose vertex set is $X$ and where two vertices are joined provided the corresponding elements of $\Gamma$ are incident.
A {\it flag} is a set of pairwise incident elements of $\Gamma$, i.e. a clique of its incidence graph.
The {\it type} of a flag $F$ is $\{t(x) : x \in F\}$.
A {\it chamber} is a flag of type $I$.
An element $x$ is {\em incident} to a flag $F$ and we write $x*F$ for that, when $x$ is incident to all elements of $F$.
An incidence system $\Gamma$ is a {\it geometry} or {\it incidence geometry} if every flag of $\Gamma$ is contained in a chamber (or in other words, every maximal clique of the incidence graph is a chamber).
The {\it rank} of $\Gamma$ is the number of types of $\Gamma$, namely the cardinality of $I$.

Let $\Gamma:= (X, *, t, I)$ be an incidence system.
Given a flag $F$ of $\Gamma$, the {\em residue} of $F$ in $\Gamma$ is the incidence system $\Gamma_F := (X_F, *_F, t_F, I_F)$ where
\begin{itemize}
\item $X_F := \{ x \in X : x * F, x \not\in F\}$;
\item $I_F := I \setminus t(F)$;
\item $t_F$ and $*_F$ are the restrictions of $t$ and $*$ to $X_F$ and $I_F$.
\end{itemize}

An incidence system $\Gamma$ is {\em residually connected} when each residue of rank at least two of $\Gamma$ has a connected incidence graph. It is called {\it thin} when every residue of rank one of $\Gamma$ contains exactly two elements.  Every thin connected rank 2 geometry is an $m$-gon for some $m \in \mathbb{N}\cup\{\infty\}$. The following result reduces the thinness test to residues of rank two.
\begin{lemma} \cite{BuekCohen} An incidence geometry of rank at least two is thin if and only if all of its rank two residues are
thin.
\end{lemma}
A {\em hypertope} is a thin incidence geometry which is residually connected.

Let $\Gamma$ be a hypertope.  The \emph{Buekenhout diagram} of $\Gamma$ is a graph whose vertices are the elements of $I$ and with 
an edge $\{i,j\}$ with label $m$ whenever $m\neq 2$. A \emph{polytope} is a hypertope with linear Buekenhout diagram.

Let $\Gamma:=(X,*, t,I)$ be an incidence system.
An {\em automorphism} of $\Gamma$ is a mapping
$\alpha:(X,I)\rightarrow (X,I):(x,t(x)) \mapsto (\alpha(x),t(\alpha(x))$
where
\begin{itemize}
\item $\alpha$ is a bijection on $X$ inducing a bijection on $I$;
\item for each $x$, $y\in X$, $x*y$ if and only if $\alpha(x)*\alpha(y)$;
\item for each $x$, $y\in X$, $t(x)=t(y)$ if and only if $t(\alpha(x))=t(\alpha(y))$.
\end{itemize}
An automorphism $\alpha$ of $\Gamma$ is called {\it type preserving} when for each $x\in X$, $t(\alpha(x))=t(x)$.
An incidence system $\Gamma$ is {\em flag-transitive} if $Aut_I(\Gamma)$ is transitive on all flags of a given type $J$ for each type $J \subseteq I$.
Finally, an incidence system $\Gamma$ is {\em regular}  if $Aut_I(\Gamma)$ acts regularly on the chambers (i.e. the action is semi-regular and transitive).
A {\em regular hypertope} is a flag-transitive hypertope.

If $\Gamma$ is a regular hypertope of rank $r$ with type-set $I:= \{0, \ldots, r-1\}$, and $C$ is a chamber of $\Gamma$, the automorphism group $Aut_I(\Gamma)$ is generated by a set of \emph{distinguished generators} $\{\rho_0, \ldots, \rho_{r-1}\}$ such that $\rho_i$ maps $C$ to its $i$-adjacent chamber in $\Gamma$, that is the unique chamber $C_i$ of $\Gamma$ such that $C$ and $C_i$ differ only in their respective elements of type $i$.

Given an incidence system $\Gamma$ and a chamber $C$ of $\Gamma$, we may associate to the pair $(\Gamma,C)$ a pair consisting of a group $G$ and a set $\{G_i : i \in I\}$ of subgroups of $G$ where $G := Aut_I(\Gamma)$ and $G_i$ is the stabilizer in $G$ of the element of type $i$ in $C$.
The following proposition shows how to reverse this construction, that is starting from a group and some of its subgroups, construct an incidence system.

Observe that in this paper all C-groups we get give hypertopes as we will show in Section~\ref{proof}.
\begin{proposition}\cite{Tits56}\label{tits}
Let $n$ be a positive integer
and $I:= \{0,\ldots ,r-1\}$ a finite set.
Let $G$ be a group together with a family of subgroups ($G_i$)$_{i \in I}$, $X$ the set consisting of all cosets $G_ig$, $g \in G$, $i \in I$ and $t : X \rightarrow I$ defined by $t(G_ig) = i$.
Define an incidence relation $*$ on $X\times X$ by :
\begin{center}
$G_ig_1 * G_jg_2$ iff $G_ig_1 \cap G_jg_2$ is non-empty in $G$.
\end{center}
Then the 4-tuple $\Gamma := (X, *, t, I)$ is an incidence system having a chamber.
Moreover, the group $G$ acts by right multiplication as an automorphism group on $\Gamma$.
Finally, the group $G$ is transitive on the flags of rank less than 3.
\end{proposition}

If $\Gamma:= \Gamma(G;(G_i)_{i\in I})$ is a regular hypertope, its distinguished generators are the generators of the subgroups $\cap_{j \in I\backslash\{i\}}G_j$.

\begin{theorem}\cite{FLW1}\label{cgroup}
Let $I:=\{0, \ldots, r-1\}$ and let $\Gamma:= \Gamma(G;(G_i)_{i\in I})$ be a regular hypertope of rank $r$.
 The pair $(G,S)$ where $S$ is the set of distinguished generators of $\Gamma$ is a C-group of rank $r$.
\end{theorem}

Regular hypertopes with a linear Buekenhout diagram are in one-to-one correspondance with C-groups with a linear Coxeter diagram that are also called \emph{string C-groups}. 
Nevertheless from a C-group that is not string we may not get a regular hypertope. Some examples can be found in \cite{FLW1} as well as the following result. 

\begin{proposition}\cite{FLW1}\label{FTcgroup}
Let $(G,\{\rho_0, \ldots, \rho_{r-1}\})$ be a C-group of rank $r$ and let $\Gamma := \Gamma(G;(G_i)_{i\in I})$ with $G_i := \langle \rho_j | \rho_j \in S, j \in I\setminus \{i\} \rangle$ for all $i\in I:=\{0, \ldots, r-1\}$.
If $\Gamma$ is flag-transitive, then
$\Gamma$ is thin, residually connected and regular.
\end{proposition}

The following result gives a way to check whether or not $\Gamma$ is a flag-transitive geometry. See also Dehon \cite{Dehon94}.

\begin{theorem}\cite{BH91}\label{hermandft}
 Let $\mathcal{P}(I)$ be the set of all the subsets of $I$ and let $\alpha: \mathcal{P}(I) \setminus\{ \emptyset \} \to I$ be a function such that $\alpha (J) \in J$ for every $J \subset I$, $J \neq \emptyset$. The geometry $\Gamma := \Gamma(G;(G_i)_{i\in I})$ is flag-transitive if and only if, for every $J \subset I$ such that $\vert J \vert \geq 3$, we have \[ \bigcap_{j \in J - \alpha (J)} (G_j G_{\alpha(J)}) = \Bigg( \bigcap_{j \in J - \alpha (J)} G_j \Bigg) G_{\alpha(J)} \]
\end{theorem}

A proof of this result is also available in ~\cite[Theorem 1.8.10]{BuekCohen}.
We prefer to use the following result to check flag-transitivity as it is much easier to apply in our case.
\begin{theorem}\label{FTlee2}\cite{FTLee}
 Let $\Gamma(G,\{G_0,\ldots,G_{r-1}\})$ be a flag-transitive coset geometry of rank $r$, let $H$ be a subgroup of $G$ and let $\Gamma'(H,\{G_0 \cap H, \ldots,$ $G_{r-1} \cap H\})$ be a flag-transitive geometry. Then the incidence systems $\Gamma_{(ij)}(G,$ $\{G_i,G_j,H\})$ are flag-transitive geometries $\forall ~ 0\leq i,j \leq r-1$, $i \neq j$ if and only if  the incidence system $\Gamma''(G,\{G_0,\ldots,G_{r-1},H\})$ is a flag-transitive geometry.
\end{theorem}

\subsection{Permutation representation graphs}
In what follows, $\Gamma:= (G, \{\rho_0,\ldots,\rho_{r-1}\})$ is a group generated by involutions, $G$ is of permutation degree $n$
 and the maximal parabolic subgroups of $\Gamma$ are the subgroups $\Gamma_i:=\langle \rho_j\,|\, j\neq i\rangle$, $i=0,\ldots, r-1$.
Let $J\subseteq I:= \{0,\ldots,r-1\}$, we define $\Gamma_J:= \langle \rho_i: i \in I\setminus J\rangle$.
If $J=\{i_1,\ldots,i_k\}$ we write 
$\Gamma_{i_1,\ldots,i_k}$, omitting the set brackets. 

The \emph{permutation representation graph} $\mathcal{G}$ of $\Gamma$ is the graph with $n$ vertices and an $i$-edge $\{a,b\}$ whenever $a=\rho_i b$. We denote by $\mathcal{G}_I$ the subgraph of $\mathcal{G}$ with $n$ vertices and  with the edges of $\mathcal{G}$ that have labels in $I$. As  $\Gamma$ is generated by involutions, $\mathcal{G}_{\{i\}}$ is a matching.

A \emph{fracture graph} $\mathcal{F}$ of $\Gamma$ is a subgraph of $\mathcal{G}$ containing all vertices of $\mathcal{G}$ and  one edge of each label, chosen in such a way that each $i$-edge joins two vertices $a_i$ and $b_i$ that are in distinct $\Gamma_i$-orbits. Fracture graphs play an important rule when every $\Gamma_i$ is intransitive as in that case a fracture graph has exactly $r$ edges and is a forest with $c$ components when $r=n-c$ (see \cite{extension}). When we want to represent a fracture graph $\mathcal{F}$ it is  convenient to distinguish edges in $\mathcal{F}$ from edges that are not in $\mathcal{F}$. For that reason dashed edges will be used for edges in  $\mathcal{G}\setminus \mathcal{F}$. 

\section{C-groups of rank $n-1$ for $S_n$}\label{n-1}

In \cite{CLP} the authors give an enumeration of the inductively minimal geometries of any rank by exhibiting a
correspondence between the inductively minimal geometries of rank $n-1$ and the trees with $n$
vertices. More precisely, the line graph of a tree is the diagram of a inductively minimal geometry and vice-versa.
We recall that the\emph{ line graph} of a given graph $\mathcal{G}$ is defined as follows:
\begin{itemize}
\item the vertices are the edges of $\mathcal{G}$;
\item two vertices of the line graph are adjacent if they have a common vertex in $\mathcal{G}$.
\end{itemize}

In \cite{CaCa:ta} the authors answer the 3rd question of Section 3 of \cite{CLP}, showing that the tree corresponding to a  inductively minimal geometry is just the permutation representation of an independent set of size $n-1$ for $S_n$, when $n\geq 7$. We summarised their result in Theorem~\ref{CaCa}.

We note that the Coxeter diagram $\Delta$ of a  inductively minimal geometry satisfies the following three properties:
\begin{itemize}
\item  $\Delta$ has no chordless cycle of length greater then 3;
\item every edge of $\Delta$ is in a unique maximal clique;
\item each vertex of $\Delta$ is either in one or in two maximal cliques.
\end{itemize}
A graph satisfying these three conditions is called an \emph{IMG diagram} in~\cite{CLP}. The Coxeter diagram of a C-group $\Gamma$ of rank $n-1$  is an IMG diagram. In the following proposition we prove that each triangle of the Coxeter diagram  of $\Gamma$ corresponds to the finite geometric group $[111]^2$.

\begin{proposition}\label{preS}
For $n\geq 7$,  $\Gamma:=(S_n,\{\rho_0,\ldots, \rho_{n-2}\})$ is a C-group of rank $n-1$ if and only if $\Gamma$ is abstractly defined by the relations corresponding to its Coxeter diagram, that is an IMG diagram, and a relation $(\rho_i\rho_j\rho_i\rho_k)^2=1$ whenever $\{\rho_i,\rho_j,\rho_k\}$ is a triangle of the Coxeter diagram. 
\end{proposition}
 \begin{proof}
 Suppose that $\{\rho_i,\rho_j,\rho_k\}$ is a triangle of the Coxeter Diagram of $\Gamma$. Then $\mathcal{G}_{\{i,j,k\}}$ is as follows. 
$$ 
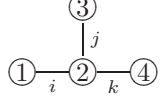
\xymatrix@-1pc{& *+[o][F]{3} &\\
 *+[o][F]{1} \ar@{-}[r]_i  & *+[o][F]{2}  \ar@{-}[r]_k\ar@{-}[u]_j&*+[o][F]{4}} $$
For the above numbering of the vertices we have $\rho_i\rho_j\rho_i\rho_k=(1\,3)\,(2\,4)$, therefore $(\rho_i\rho_j\rho_i\rho_k)^2=1$.

Now let us prove that the relations of the Coxeter diagram plus the relations of the form $(\rho_i\rho_j\rho_i\rho_k)^2=1$ corresponding to a triangle of the Coxeter Diagram, are sufficient to give a  presentation of $\Gamma$. 

First when the Coxeter diagram of $\Gamma$ has no triangles, $\Gamma$ is the finite Coxeter group $A_{n-1}$, thus in that case the Coxeter diagram gives all the relators for a presentation of $\Gamma$. We now proceed by induction,  assuming that the proposition holds whenever the Coxeter diagram has at most $t$ triangles. 

Suppose that the Coxeter diagram of $\Gamma$ has exactly $t$ triangles. 
Pick any vertex $v$ of $\mathcal{G}$ of degree one and let $w$ be  the vertex of degree at least three closer to $v$. Now let $z$ be another vertex of degree one. Consider the labelling of the edges as in the following figure. 
$$ 
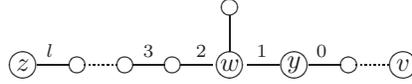
\xymatrix@-1pc{&&&& *+[o][F]{} &\\
*+[o][F]{z} \ar@{-}[r]^l&*+[o][F]{} \ar@{.}[r]& *+[o][F]{} \ar@{-}[r]^3&*+[o][F]{}  \ar@{-}[r]^2 & *+[o][F]{w}  \ar@{-}[r]^1\ar@{-}[u]&*+[o][F]{y} \ar@{-}[r]^{0}&*+[o][F]{} \ar@{.}[r]&*+[o][F]{v}} $$
Removing the $1$-edge $\{w,y\}$ of $\mathcal{G}$ and replacing it by the $1$-edge $\{y,z\}$ we obtain another graph $\mathcal{D}$ that is also a tree . 
$$ 
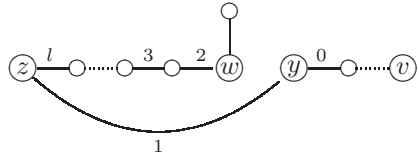
\xymatrix@-1pc{&&&& *+[o][F]{} &\\
*+[o][F]{z}  \ar@{-}@/_2pc/[rrrrr]_1\ar@{-}[r]^l&*+[o][F]{} \ar@{.}[r]& *+[o][F]{} \ar@{-}[r]^3&*+[o][F]{}  \ar@{-}[r]^2 & *+[o][F]{w}  \ar@{-}[u]&*+[o][F]{y} \ar@{-}[r]^0&*+[o][F]{} \ar@{.}[r]&*+[o][F]{v}} $$
By Theorem~\ref{CaCa}, $\mathcal{D}$ is the permutation representation graph of a C-group $\Delta$ of rank $n-1$ for $S_n$. 
By construction the Coxeter diagram of  $\Delta$ has less triangles than the Coxeter diagram of $\Gamma$. Indeed we reduced the degree of $w$ increasing the degree of a vertex of degree one, thus we did not create another triangle, as only vertices of degree at least three correspond to triangles of the Coxeter diagram.

Let us denote by $\{\alpha_0,\ldots,\alpha_{n-2}\}$  the generating set of $\Delta$. By induction apart from the relations given by the Coxeter diagram of $\Delta$, that is the line graph  of $\mathcal{D}$  following  Theorem~\ref{CaCa},  we have the relations $(\alpha_i\alpha_j\alpha_i\alpha_k)^2=1$ for each triangle  $\{\alpha_i,\alpha_j,\alpha_k\}$ of the Coxeter Diagram.

In what follows we rewrite the relations of the presentation of $\Delta$ in terms of the $\rho_i$'s to get a presentation of $\Gamma$.
 According to the way the permutation representation of $\Delta$ was obtained from the permutation representation of $\Gamma$ we have that $\rho_i=\alpha_i$ for $i\neq 1$ and $\rho_1=\alpha_1\,^{\alpha_2\ldots\alpha_{l}}$. Using this substitution, all the relations not involving $\alpha_1$ give either relations of the Coxeter diagram of $\Delta$ or relations of the form $(\rho_i\rho_j\rho_i\rho_k)^2=1$ whenever $\{\rho_i,\rho_j,\rho_k\}$ is a triangle of the Coxeter diagram of $\Delta$ (and of $\Gamma$).
Therefore we need to consider only the relations of $\Delta$ involving $\alpha_1$.

First if $i$ is such that the  $i$-edge of $\mathcal{D}$ is not incident to a vertex of the path between $w$ and $z$, then $\alpha_i $ commutes with $h=\alpha_2\ldots\alpha_{l}$, hence
 $\rho_1\rho_i=h^{-1}\alpha_1h\alpha_i=(\alpha_1\alpha_i)^h$. Therefore  $\alpha_1\alpha_i$ and $\rho_1\rho_i$ have the same order.

Now suppose that $i$ is the label of an edge $e$ incident to a vertex of the path from $w$ to $z $. We deal with the following cases separately: (1) $e$ is incident to $w$ and $i\neq l$, (2) $e$ is an edge of the path and is not incident to $w$ and $i\in\{2,\ldots, l-1\}$, (3) $e$ is an edge of the path and $i=l$ and (4) $i\notin\{1,\ldots,l-1\}$ but is a label of  an edge incident to the path from $w$ to $z $. Let $g_i:=\rho_i\ldots\rho_{l}$ for $i\in\{2,\ldots,l-1\}$, $g_l:=\rho_l $ and $g_{l+1}:=1$.

(1) In this case $(\alpha_1\alpha_i)^2=1$. We have $\alpha_1\alpha_i=g_3^{-1}\rho_2\rho_1\rho_2g_3\rho_i=g_3^{-1}\rho_2\rho_1\rho_2\rho_i g_3$. Thus $(\alpha_1\alpha_i)^2=1 \Leftrightarrow (\rho_2\rho_1\rho_2\rho_i)^2=1$. In this case we get the relation corresponding to the triangles $\{\rho_1,\rho_2,\rho_i\}$ of the Coxeter diagram of $\Gamma$.

(2) Here $(\alpha_1\alpha_i)^2=1$ and $\alpha_1\alpha_i=g_2^{-1}\rho_1g_2\rho_i=g_2^{-1}\rho_1\rho_2\ldots \rho_{i-1}(\rho_i\rho_{i+1}\rho_i)g_{i+2}=g_2^{-1}\rho_1\rho_2\ldots \rho_{i-1}(\rho_{i+1}\rho_i\rho_{i+1})g_{i+2}=g_2^{-1}\rho_1\rho_{i+1}\rho_2\ldots\rho_{i-1} \rho_i\rho_{i+1}g_{i+2}=(\rho_1\rho_{i+1})^{g_2}$. Thus $(\alpha_1\alpha_i)^2=1\Leftrightarrow (\rho_1\rho_{i+1})^2=1$, $i\in\{2,\ldots, l-1\}\Leftrightarrow (\rho_1\rho_i)^2=1$, $i\in\{3,\ldots, l\}$. In this case we obtain a relation implicit in the Coxeter diagram of $\Gamma$.

(3) Let first  $l\neq 2$. We have $\alpha_1=\rho_l\,^{\rho_{l-1}\ldots\rho_2\rho_1}$ 
Hence $\alpha_1\alpha_l=\rho_l\,^{\rho_{l-1}\ldots\rho_2\rho_1}\rho_l=\rho_1\rho_2\ldots\rho_{l-2}(\rho_{l-1}\rho_l\rho_{l-1}\rho_l)\rho_{l-2}\ldots\rho_2\rho_1=(\rho_l\rho_{l-1})^{\rho_{l-2}\ldots\rho_2\rho_1}$ thus $(\alpha_1\alpha_l)^3=1\Leftrightarrow (\rho_l\rho_{l-1})^3=1$. When $l= 2$, $\alpha_1=\rho_1^{\rho_2}=\rho_2^{\rho_1}$ hence $(\rho_1\rho_2)^3=1$. In any case we obtain a relation implicit in the Coxeter diagram of $\Gamma$. 

(4)  Suppose that $i$ is incident to the $j$-edge and to the $(j-1)$-edge for some $j\in \{3,\ldots,l\}$. 
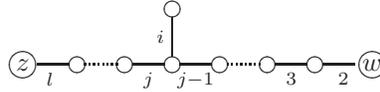
$$ \xymatrix@-1pc{&&&*+[o][F]{} \ar@{-}[d]_i&&\\
*+[o][F]{z} \ar@{-}[r]_l&*+[o][F]{} \ar@{.}[r]&*+[o][F]{} \ar@{-}[r]_j&*+[o][F]{}\ar@{-}[r]_{j-1}&*+[o][F]{} \ar@{.}[r] &*+[o][F]{} \ar@{-}[r]_3&*+[o][F]{}  \ar@{-}[r]_2 & *+[o][F]{w}  } $$
In this case we have that $(\alpha_{j-1}\alpha_j\alpha_i\alpha_j)^2=1$  thus, 
\[\begin{array}{rl}
\alpha_1\alpha_i&=g_2^{-1}\rho_1g_2\rho_i=g_2^{-1}\rho_1\rho_2\ldots\rho_{j-1}\rho_j\rho_ig_{j+1}=\\
&=g_2^{-1}\rho_1\rho_2\ldots\rho_{j-2}(\rho_{j-1}\rho_j\rho_i\rho_j)g_j=\\
&=g_2^{-1}\rho_1\rho_2\ldots\rho_{j-2}(\rho_j\rho_i\rho_j\rho_{j-1})g_j=\\
&=g_2^{-1}\rho_1\rho_j\rho_i\rho_jg_2=(\rho_1\rho_i)^{\rho_jg_2}.
\end{array}\]
Hence $(\alpha_1\alpha_i)^2=1\Leftrightarrow (\rho_1\rho_i)^2=1$. In this case we obtain a relation implicit in the Coxeter diagram of $\Gamma$.

This proves that each relation of $\Delta$ is either converted in a relation implicit in the Coxeter diagram of $\Delta$ or into a relation of the form $(\rho_i\rho_j\rho_i\rho_k)^2=1$, and the latest happens when $\{\rho_i,\rho_j,\rho_k\}$ is a triangle of the Coxeter diagram of $\Gamma$.
\end{proof}

\section{Bounding the C-rank of a transitive maximal parabolic subgroup of a C-group for $S_n$}\label{transitive}

In this section we prove that if $\Gamma:=(S_n,\{\rho_0,\ldots, \rho_{r-1}\})$ is a C-group of rank $r$ and  $\Gamma_i$ is transitive for some $i\in\{0,\ldots,r-1\}$, then $r\leq n-3$ when $n\geq 9$. We first deal with the case where $\Gamma_i$ is a primitive group of degree $n$ for some $i\in\{0,\ldots,r-1\}$. Let us recall the  following result that  gives a bound for the order of a primitive group of given degree.

\pagebreak
\begin{theorem}\cite{Mar2002}\label{maroti}
Let $G$ be a primitive group of degree $n$ which is not $S_n$ nor $A_n$. Then one of the following possibilities occurs:
\begin{enumerate}
\item For some integers $m, k, l$ we have $n =(^m_k)^l$, and $G$ is a subgroup
of $S_m\wr S_l$, where $S_m$ is acting on $k$-subsets of $\{1,\ldots, m\}$;
\item $G$ is $M_{11}, M_{12}, M_{23}$ or $M_{24}$ in its natural $4$-transitive action;
\item $\displaystyle |G|\leq n.\prod_{i=0}^{\lfloor\log _2 n-1\rfloor}(n-2^i)$
\end{enumerate}
\end{theorem}
Now we establish that the C-rank of $\Gamma_i$ is at most $n-4$ when $\Gamma_i$ is primitive and $n\geq 9$.
\begin{lemma} \label{pri<=n-4}
The C-rank of a C-group $(\Gamma, \{\rho_0,\ldots ,\rho_{r-1}\})$ where $\Gamma$ is a primitive group of degree $n\geq 9$, not isomorphic to $A_n$ or $S_n$, is at most $n-4$.
\end{lemma}
\begin{proof}
We consider separately the three possibilities given by Theorem~\ref{maroti}.

In case (1)  when $l\geq 2$ we have 
$n-4 = (^m_k)^l-4\geq m^l-4\geq ml-2\geq r$. When $l=1$ we have $k\geq 2$, $m\leq n/2$ and the group is a subgroup of $S_m$ or $A_m$, so its
rank is at most $m-1$, much smaller than $n-4$.

In case (2) we have to consider the groups $M_{11}, M_{12}, M_{23}$ or $M_{24}$. The maximal length of a chain of subgroups of $M_{11}$, $M_{23}$ or $M_{24}$ is 7, 11 and 14 resp. (see \cite{Whis00}). 
If $\Gamma$ is isomorphic to $M_{12}$ then the rank of $\Gamma$ is at most 9 \cite{Whis00}. 
Suppose that the C- rank of  $\Gamma$ is $9$. Then one of the following subgroups of $M_{12}$, namely $M_{11}$ or $P\Gamma L(2,9)$, has to have C- rank  8.
As the rank of $M_{11}$ is at most 7 (see~\cite{Whis00}), $P\Gamma L(2,9)$ should have C- rank  $8$. 
If this is so, one the following groups, $PGL(2,9)$, $S_6$ or $M_{10}$ have C- rank  7. The latter is not generated by involutions, $S_6$ is known to have C- rank  at most 5 and $PGL(2,9)$ has C- rank  at most 3 (see \cite{DJT}), thus we get a contradiction. Hence the C- rank  of $M_{12}$ is at most 8.

In case (3) we have that the chain length is bounded by $\log_2 \left[n.\prod_{i=0}^{\lfloor\log _2 n-1\rfloor}(n-2^i)\right]$ that is at most $n-4$ for $n\geq 26$. 
For the primitive groups of degree 9 (see for instance~\cite{BL96}), we readily see that the only groups that can be generated by involutions are $3^2:D_8$, $AGL(2,3)$ and $PSL(2,8)$. The group $3^2:D_8$ has chains of maximum length 5 in its subgroup lattice.
An exhaustive computer search with {\sc Magma}~\cite{BCP97} gives C-rank 4 for $AGL(2,3)$. The C-rank of $PSL(2,8)$ is known to be 3 by~\cite{DJT}.
The primitive groups $G$ with degree between 10 and 25 for which $\log_2(|G|)>n-4$ are listed in Table~\ref{exceptions} where $r$ is the C-rank of the corresponding group. 
\begin{table}
\begin{center}
\begin{tabular}{cll|cll|cll}
$n$& $G$ &$r$&$n$& $G$ &$r$&$n$& $G$ &$r$\\
\hline
$10$& $PSL_2(9)$& $3$ \cite{DJT}&         $12$&$PSL_2(11)$& $4$  \cite{DJT}&      $16$&$2^4:A_6$&$\leq 10$\\      
         &$S_6$ & $5$ \cite{Whis00}&                     &$PGL_2(11)$& $4$ \cite{DJT}&                     &$2^4:S_6$&$\leq 11$\\                      
        &$PGL_2(9)$ & $4$ \cite{DJT}&           $13$&$PSL_3(3)$&  $\leq 8$ &                     &$2^4:A_7$&$\leq 11$\\                 
        &$M_{10}$ & 0 &                 $14$& $PGL_2(13)$&$3$ \cite{DJT}&                    &$AGL_4(2)$& $\leq 12$\\
        &$P\Gamma L_2(9)$&$\leq 6$&                       $15$&$PSL_4(2)$&$\leq 9$ &              $22$&$M_{22}:2$&$\leq 13$\\
        
$11$&$PSL_2(11)$& $4$ \cite{DJT}&&&&&\\
                           
\end{tabular}
\end{center}
\caption{The groups not failing the chain length bound.}\label{exceptions}
\end{table}
It is well known that $M_{10}$ cannot be generated by involutions, hence its C-rank is 0.
For $P\Gamma L(2,9)$, we already showed above that the C- rank  is at most $6$.
A chain of subgroups in the subgroup lattice of $PSL(3,3)$ has length 8 at most, hence the C-rank of $PSL(3,3)$ is 8 at most.
A chain of subgroups in the subgroup lattice of $PSL(4,2)$ has length 9 at most, hence the C-rank of $PSL(4,2)$ is 9 at most.
A chain of subgroups in the subgroup lattice of $2^4:A_7$ or $2^4:S_6$ has length 11 at most, hence the C-rank of these groups is 11 at most.
A chain of subgroups in the subgroup lattice of $2^4:A_6$ has length 10 at most, hence the C-rank is 10 at most.
A chain of subgroups in the subgroup lattice of $M_{22}:2$ has length 13 at most, hence the C-rank is 13 at most.
For $AGL(4,2)$, a chain of subgroups may have up to length 13 but a computer search shows that none of the groups at Level 5 in the subgroup lattice that could have a C-group representation with a non-empty graph as diagram have C- rank  5. So $AGL(4,2)$ cannot be of C- rank 13.
\end{proof}

\begin{lemma} \label{neverAnM12AGL}
If  $\Gamma:=(S_n,\{\rho_0,\ldots, \rho_{r-1}\})$  is a C-group of rank $n-2$ with $n\geq 3$, then $\Gamma_i$ is not isomorphic to  $A_n$.
\end{lemma}
\begin{proof}
Suppose without loss of generality that $\Gamma_0$ is isomorphic to $A_n$.
Then, since $\Gamma$ satisfies the intersection property, we have $\Gamma_0 \cap \langle \rho_0,\rho_i \rangle = \langle \rho_i \rangle$ for all $i = 1,\ldots, n-3$.
But then $\rho_i^{\rho_0} = \rho_i$ for all $i$ which implies that $S_n \cong A_n\times C_2$, a contradiction.
\end{proof}

\begin{proposition}\label{imp}
Let $n\geq 9$. If  $\Gamma:=(S_n,\{\rho_0,\ldots, \rho_{r-1}\})$  is a C-group of rank $r$ and  $\Gamma_i$  is transitive imprimitive for some $i\in \{0,\ldots, r-1\}$, then $r\leq n-3$.
\end{proposition}

\begin{proof}
The case $n=9$ is dealt with by performing an exhaustive computer search using {\sc Magma}. We now assume $n\geq 10$.
Suppose that $\Gamma_i$ is embedded into $S_k\wr S_m$ with $k$ being maximal.
Let $S=\{\rho_j\,|\, j\in I, j\neq i\}$ and denote by $M$ the subset of $S$ generating the block action.
We have that $|M|\leq m-1$ (see \cite{Whis00}).
We can use the elements of $M$ to undo the block action  of the elements of $N:=S\setminus M$ (by right multiplication), and as shown in Lemma 3 of 
\cite{Whis00}, we get an independent set $\bar{S}=M\cup \bar{N}$ with $\bar{N}$ fixing the blocks setwisely.
Consider the action of $\bar{N}$ on the blocks. 

First  suppose there is no ordering on the blocks such that $\bar{N}$ acts as $S_k$ on the first $m-2$  blocks.
In that case $|S|\leq (m-1)+m(k-1)-3 = n-4.$ Thus $r\leq n-3$.

Now assume there exists an ordering such that the action of $\bar{N}$ on the first $m-2$ blocks is $S_k$. 
We now use the following argument used in the proof of Proposition 4 of \cite{Whis00}:
if $B\subset \bar{N}$ generate the action on the first $i$ blocks  for  $i\leq m-2$, then at most one element of $\bar{N}$ needs to be added to $M\cup B$ to generate the block action on $(i+1)$ blocks.
 Therefore,
 $$|S|\leq (m-1)+(k-1)+(m-2)+(k-2)=2m+2k-6.$$
 Now $2m+2k-6\leq km-4 \Leftrightarrow(m-2)(k-2)\geq 2$ which is true for $n\geq 10$
and $m,k\neq 2$.

Let us consider the case  $k=2$. If $ |\bar{N}|\leq 1$ then $r\leq m=\frac{n}{2}\leq n-4$, thus it may be assumed that $ |\bar{N}|\geq 2$. In that case there exists $j\neq i$ such $\Gamma_{i,j}$ is transitive. First if $\Gamma_j$ is primitive then by Lemmas~\ref{pri<=n-4} and \ref{neverAnM12AGL} $r\leq n-3$. If $\Gamma_j$ is imprimitive by assumption $\Gamma_j$ cannot be embedded  into a wreath product with  more than two blocks. Hence  $\Gamma_{i,j}$ is embedded into $S_2\wr S_{\frac{n}{2}}$ for two distinct blocks systems. Now we can use the argument used in \cite{Corr} to prove that $\Gamma_{i,j}$ must be a dihedral group. Therefore $r-2\leq \log _2(2n)\leq n-5$ for $n\geq 10$.
 
Now suppose that $m=2 $ and assume the action of  $\bar{N}$ is not $S_k$ neither in block 1 nor in block 2, for otherwise $r-1\leq 1+(k-1)+1\leq 2k-3=n-3$ for $n\geq 10$. 
Consider two subsets of $\bar{N}$, $A$ and $B$, generating independently the block action of $\bar{N}$ on block 1 and on block 2, respectively.
If  $A\cap B\neq \emptyset$ then $|\bar{N}|\leq|A|+|B|-1$, thus
$$|S|\leq 1+ 2(k-2)-1=n-4.$$
Suppose that  $A\cap B=\emptyset$ and let $M=\{\tau\}$. As the conjugation by the permutation of $\tau$ defines an isomorphism between the group action on block 1  and the group action on block 2, these groups  must both be transitive on the respective blocks, otherwise $\Gamma_i$ is itself intransitive, a contradiction. 
Then if  any element $\alpha_j\in\bar{N}=A\cup B$ is removed from the generating set $M\cup\bar{N}$, we still get a transitive group
$\Gamma_{i,j}$. If $\Gamma_j$ is primitive then $r\leq n-3$.
Suppose that $\Gamma_j$ is imprimitive.  If $\Gamma_j$ is embedded into a block system with more than two blocks of size greater than two, then we have proved previously that $r\leq n-3$. Thus we need to consider only two cases: $\Gamma_j$  is embedded into $S_{\frac{n}{2}}\wr S_2$ and $\Gamma_j$  is embedded into $S_2\wr S_{\frac{n}{2}}$.

Let us first consider that $\Gamma_j$  is embedded into $S_{\frac{n}{2}}\wr S_2$. 
The block systems of the embeddings of $\Gamma_j$ and $\Gamma_i$ need to be different, otherwise $\Gamma$ is itself embedded into  $S_{\frac{n}{2}}\wr S_2$, a contradiction.
Thus $\Gamma_{i,j}$ is embedded into $S_{\frac{n}{2}}\wr S_2$ for two distinct block systems $\{X,X\tau\}$ and $\{Y,Y\delta\}$. 
Then  $\Gamma_{i,j}$ is also embedded into $S_{\frac{n}{4}}\wr S_4$ with the block system being $\{X\cap Y\}, \{X\cap Y\delta\},\{X\tau\cap Y\},\{X\tau\cap Y\delta\}\}$ with $\{\tau, \delta\}$ generating the block action. Now let $M=\{\tau, \delta\}$,  let $N$ be the set of the remaining generators of $\Gamma_{i,j}$ and consider the set $\bar{N}$ as before.
If there is no ordering on the blocks such that $\bar{N}$ acts as $S_{\frac{n}{4}}\times S_{\frac{n}{4}}$ in the first two blocks we have 
$|S|-1\leq 2+ 4(n/4-1)-3\Rightarrow |S|\leq n-4$.
Now suppose the contrary. Hence we have $|S|-1\leq 2+ (n/4-1)+2+ (n/4-2)\Rightarrow |S|\leq n-3 $ for $n\geq 10$.
Suppose that $\Gamma_j$  is embedded into $S_2\wr S_{\frac{n}{2}}$. 
Now we may assume that for this embedding $|\bar{N}|\leq 1$ otherwise there exist $k\neq j,i$ such that $\Gamma_k$ is transitive and then we get one of the situations considered previously. Hence $r-1\leq 1+ \frac{n}{2}-1\leq n-4$ for $n\geq 8$.
\end{proof}

\section{Intransitive maximal parabolic subgroups of a C-group of rank $n-2$ for $S_n$}\label{intransitive}

Let $n\geq 9$ and  $\Gamma:= (S_n, \{\rho_0,\ldots, \rho_{n-3}\})$ be a C-group of rank $n-2$.
Suppose that $\Gamma_i$ is intransitive for every $i\in\{0,\ldots, n-3\}$.
Let $\mathcal{G}$ be the permutation representation graph of $\Gamma$ and $\mathcal{F}$ be a fracture graph of $\Gamma$. 
We recall the following lemma that will be useful in later proofs.

\begin{lemma}\cite[Lemmas 3.2 and 3.5]{extension}\label{lemma3.1}
$\mathcal{F}$ has exactly 2  connected components. 
Moreover, the edges of $\mathcal{G}$ that are not in $\mathcal{F}$ must connect vertices in different connected components of $\mathcal{F}$.
\end{lemma}

In the next lemma we prove that the connection between the two components of the fracture graph is made either by a single edge, a multiple edge or an alternating square.

\begin{lemma}\label{connection}
If $\{a,b\}$ and $\{c,d\}$ are edges of $\mathcal{G}\setminus\mathcal{F}$ with labels $i$ and $j$ respectively then: 
\begin{enumerate}
\item $\{a,b\}\cap\{c,d\}\neq \emptyset$, particularly $i\neq j$;
\item If $\{a,b\}\neq \{c,d\}$ then $i$ and $j$ are labels of edges of a square with alternating labels $i$ and $j$ as in the following figure.
$$ 
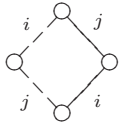
\xymatrix@-1pc{ & *+[o][F]{}   \ar@{-}[dr]^j & \\
*+[o][F]{}  \ar@{--}[ur]^i \ar@{--}[dr]_j& & *+[o][F]{}\ar@{-}[dl]^i \\
&*+[o][F]{}&  }$$
\end{enumerate}
\end{lemma}
\begin{proof}
Suppose $\{a,b\}\cap\{c,d\}=\emptyset$. Then, by Lemma~\ref{lemma3.1}, there exists a $k$-edge $\{e,f\}$ of $\mathcal{F}$ with $i\neq k\neq j$ and with $e$ and $f$ in the same connected component of $\Gamma_k$ (see picture below), a contradiction.
$$ 
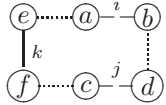
\xymatrix@-1pc{ *+[o][F]{e}   \ar@{.}[r] \ar@{-}[d]^k& *+[o][F]{a} \ar@{--}[r]^i  & *+[o][F]{b}  \ar@{.}[d]\\
*+[o][F]{f}   \ar@{.}[r]& *+[o][F]{c} \ar@{--}[r]^j & *+[o][F]{d}\ar@{.}[u]  }$$
Now suppose that $a=c$ and $b\neq d$. Then to avoid a contradiction as above, $\{a,b\}$ and $\{a,d\}$ must be edges of a $\{i,j\}$-square with $a$ in one component of $\mathcal{F}$ and the other vertices of the square in the other component as in the following figure.
$$ 
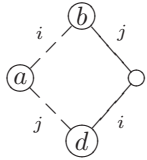
\xymatrix@-1pc{ & *+[o][F]{b}   \ar@{-}[dr]^j & \\
*+[o][F]{a}  \ar@{--}[ur]^i \ar@{--}[dr]_j& & *+[o][F]{}\ar@{-}[dl]^i \\
&*+[o][F]{d}&  }$$
\end{proof}
The following lemma shows that the distinguished generators of $\Gamma$ are either transpositions or 2-transpositions.
\begin{lemma}\label{distance}
Let $n\geq 9$. If $\mathcal G$ has at least two $i$-edges, then $\mathcal G$ has exactly two $i$-edges and the distance between a pair of $i$-edges of $\mathcal{G}$ is one.
\end{lemma}
\begin{proof}
Suppose $\{a,b\}$ and $\{c,d\}$ are $i$-edges of $\mathcal{G}$. By Lemma~\ref{connection} (a), there are at most two $i$-edges in $\mathcal G$. Let $\{a,b\}$ be the $i$-edge in $\mathcal{F}$. 
If the two $i$-edges are in a square, as in Lemma~\ref{connection} (2), then the distance between them is one, as wanted. 
Now consider that the $i$-edges are not in a square and are not at distance one. 
Suppose that $j$ and $l$ are labels of edges of a path from  $\{a,b\}$ to $\{c,d\}$, as follows. 
$$\xymatrix@-1pc{ *+[o][F]{a}   \ar@{-}[r]^i  & *+[o][F]{b} \ar@{-}[r]^l  & *+[o][F]{}\ar@{.}[r] &*+[o][F]{} \ar@{-}[r]^j &*+[o][F]{c} \ar@{--}[rr]^i&& *+[o][F]{d}}$$
By Lemmas~\ref{lemma3.1} and~\ref{connection}, there are four possibilities for the graph $\mathcal{G}_{i,j,l}$:
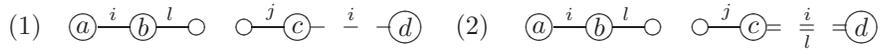
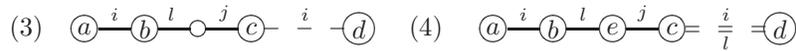
$$ (1)\quad \xymatrix@-1pc{ *+[o][F]{a}   \ar@{-}[r]^i  & *+[o][F]{b} \ar@{-}[r]^l  & *+[o][F]{} &*+[o][F]{} \ar@{-}[r]^j &*+[o][F]{c} \ar@{--}[rr]^i& &*+[o][F]{d}} \quad\mbox{ (2) }\quad\xymatrix@-1pc{ *+[o][F]{a}   \ar@{-}[r]^i  & *+[o][F]{b} \ar@{-}[r]^l  & *+[o][F]{}  &*+[o][F]{} \ar@{-}[r]^j &*+[o][F]{c} \ar@{==}[rr]^i_l& &*+[o][F]{d}}$$

$$(3) \quad \xymatrix@-1pc{ *+[o][F]{a}   \ar@{-}[r]^i  & *+[o][F]{b} \ar@{-}[r]^l  & *+[o][F]{} \ar@{-}[r]^j &*+[o][F]{c} \ar@{--}[rr]^i& &*+[o][F]{d}} \quad\mbox{ (4) }\quad\xymatrix@-1pc{ *+[o][F]{a}   \ar@{-}[r]^i  & *+[o][F]{b} \ar@{-}[r]^l  & *+[o][F]{e} \ar@{-}[r]^j &*+[o][F]{c} \ar@{==}[rr]^i_l& &*+[o][F]{d}}$$

In the first and third case $(a\,b)\in\langle \rho_i,\rho_j\rangle\cap \langle \rho_i,\rho_l\rangle $ but $(a\,b)\notin \langle \rho_i\rangle$, a contradiction with the intersection property. In the second case, $\langle \rho_i,\rho_j\rangle\cap \langle \rho_j,\rho_l\rangle$ contains at least an element of order 3, a contradiction with the intersection property again. 
In case (4), as $n\geq 9$, there must be at least another edge with label $m$ in the fracture graph.
Suppose the connected component of $\mathcal F$ containing $d$ contains such an edge $m$ starting at $d$. Then,
$\langle \rho_i,\rho_m\rangle\cap \langle \rho_m,\rho_l\rangle$ contains at least an element of order 3, a contradiction with the intersection property. Hence one of the two components of the fracture graph is a single vertex.
Moreover, we may assume that the edge of label $m$ in $\mathcal F$ connects to one of $a$, $b$, $c$ or $e$.
In addition the $m$-edge does not form a double edge with one of the existing edges. 
Also, from the beginning of the proof, we know that $\rho_m$ is either a transposition or a 2-transposition and in the latter case, $\rho_m$ swaps $c$ and $d$.
It is then easily checked with {\sc Magma} that the intersection property fails for any possible $\rho_m$ for the group $\langle \rho_i,\rho_j,\rho_l,\rho_m\rangle$.
\end{proof}

\begin{lemma}\label{double}
$\mathcal{G}$ is either a tree or has exactly one cycle that is an alternating square. 
\end{lemma}
\begin{proof}
First observe that if
$\mathcal G$ is not a tree, it has exactly one cycle. By Lemma~\ref{connection} (2), this cycle must be an alternating square. Also, if
 $\mathcal{G}$ has a square then one vertex of the square is in one component of $\mathcal{F}$ and the other 3 are in the other component of   $\mathcal{F}$.
As two edges of the square are in $\mathcal{F}$ the square has no multiple edges, and it is unique.
Now suppose  that $\{a,b\}$ is a multiple edge of $\mathcal{G}$ and let $i$ and $j$ be two labels of  that edge. 
Then by Lemma~\ref{distance} there are two edges with labels $k$ and $l$ such that the graph $\mathcal{G}_{\{i,j,k,l\}}$ is one of the following:
$$ 
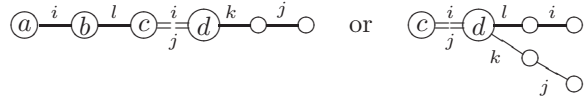
\xymatrix@-1pc{*+[o][F]{a}   \ar@{-}[r]^i& *+[o][F]{b}   \ar@{-}[r]^l  & *+[o][F]{c} \ar@{==}[r]^i_j  & *+[o][F]{d} \ar@{-}[r]^k &*+[o][F]{} \ar@{-}[r]^j &*+[o][F]{} } \quad\mbox{ or }\quad \xymatrix@-1.8pc{ *+[o][F]{c} \ar@{==}[rr]^i_j  && *+[o][F]{d} \ar@{-}[rr]^l \ar@{-}[drr]_k &&*+[o][F]{} \ar@{-}[rr]^i&&*+[o][F]{} \\
&&&&*+[o][F]{} \ar@{-}[drr]_j &&\\
&&&&&&*+[o][F]{} } $$
In both cases $(a\,b)\in \langle \rho_i,\rho_j,\rho_k\rangle\cap\langle \rho_i,\rho_j,\rho_l\rangle$ but $(a\,b)\notin \langle \rho_i,\rho_j\rangle$, a contradiction.
The rest follows from Lemma~\ref{connection}. 
\end{proof}

When $\mathcal{G}$ is a tree, let $1$ be the label of the edge between the two components of $\mathcal{F}$ and $0$ be the label of the unique edge between the two 1-edges. When 
 $\mathcal{G}$ is not a tree let $\{0,1\}$ be the labels of the alternating square, the unique cycle of $\mathcal{G}$.

\begin{lemma}\label{nosquare}
If  $\mathcal{G}$ is a tree, then $\mathcal{G}_{\{0,\ldots,r-1\}\setminus\{1\}}$ has two components, one of them having exactly two vertices.
\end{lemma}
\begin{proof}
Suppose that $k,l\geq 2$ are labels of  two edges incident to the $1$-edges, then $\mathcal{G}_{\{0,1,k,l\}}$ is one of the following graphs:
$$ 
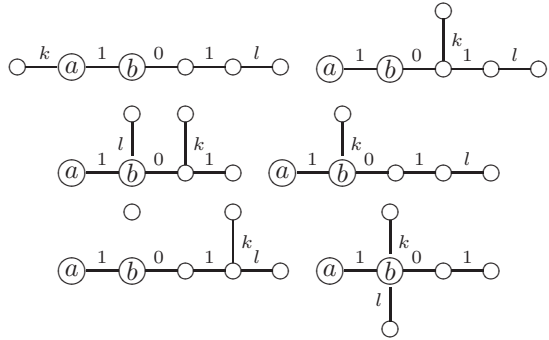
\xymatrix@-1pc{&&&\\*+[o][F]{}   \ar@{-}[r]^k& *+[o][F]{a}   \ar@{-}[r]^1 & *+[o][F]{b} \ar@{-}[r]^0  & *+[o][F]{} \ar@{-}[r]^1&*+[o][F]{} *+[o][F]{} \ar@{-}[r]^l &*+[o][F]{} }  
\quad \xymatrix@-1pc{&&*+[o][F]{}\ar@{-}[d]^k&\\
*+[o][F]{a}   \ar@{-}[r]^1& *+[o][F]{b}   \ar@{-}[r]^0& *+[o][F]{} \ar@{-}[r]^1 & *+[o][F]{} \ar@{-}[r]^l  & *+[o][F]{} }$$
$$\xymatrix@-1pc{&*+[o][F]{}\ar@{-}[d]_l&*+[o][F]{}\ar@{-}[d]^k&\\
*+[o][F]{a}   \ar@{-}[r]^1& *+[o][F]{b}   \ar@{-}[r]^0& *+[o][F]{} \ar@{-}[r]^1 & *+[o][F]{} } \quad
\xymatrix@-1pc{&*+[o][F]{}\ar@{-}[d]^k&&\\
*+[o][F]{a}   \ar@{-}[r]^1& *+[o][F]{b}   \ar@{-}[r]^0& *+[o][F]{} \ar@{-}[r]^1  & *+[o][F]{} \ar@{-}[r]^l  & *+[o][F]{} } $$
$$\xymatrix@-1pc{&*+[o][F]{}&&*+[o][F]{}\ar@{-}[d]^k\\
*+[o][F]{a}   \ar@{-}[r]^1& *+[o][F]{b}   \ar@{-}[r]^0& *+[o][F]{} \ar@{-}[r]^1 & *+[o][F]{}\ar@{-}[r]^l  & *+[o][F]{} } \quad
\xymatrix@-1pc{&*+[o][F]{}\ar@{-}[d]^k&&\\
*+[o][F]{a}   \ar@{-}[r]^1& *+[o][F]{b}   \ar@{-}[r]^0& *+[o][F]{} \ar@{-}[r]^1 & *+[o][F]{} \\
&*+[o][F]{}\ar@{-}[u]^l&&} $$
In any case $(a\,b)\in \langle \rho_0,\rho_1,\rho_k\rangle\cap\langle \rho_0,\rho_1,\rho_l\rangle$ but $(a\,b)\notin \langle \rho_0,\rho_1\rangle$, a contradiction. 
\end{proof}

\begin{lemma}\label{square}
If $\mathcal{G}$ has a square, then three vertices of the square have degree 2 in $\mathcal G$ and the fourth vertex has degree 3.
\end{lemma}
\begin{proof}
Suppose that there are two edges with labels $k$ and $l$ incident to the square, then we have the following possibilities for   $\mathcal{G}_{i,j,k,l}$
$$ \xymatrix@-1pc{ &&\\*+[o][F]{a}   \ar@{-}[r]^i \ar@{-}[d]_j  & *+[o][F]{b} \ar@{-}[r]^l \ar@{-}[d]^j  & *+[o][F]{} \\
*+[o][F]{}   \ar@{-}[r]_i  & *+[o][F]{} \ar@{-}[r]_k& *+[o][F]{}}\quad 
\xymatrix@-1pc{ &&&\\&*+[o][F]{a}   \ar@{-}[r]^i \ar@{-}[d]_j  & *+[o][F]{b} \ar@{-}[r]^l \ar@{-}[d]^j  & *+[o][F]{} \\
*+[o][F]{}   \ar@{-}[r]_k &*+[o][F]{}   \ar@{-}[r]_i  & *+[o][F]{} & }\quad \xymatrix@-1pc{ &*+[o][F]{} *+[o][F]{} \ar@{-}[d]^k&\\  *+[o][F]{a} \ar@{-}[r]^i \ar@{-}[d]_j  & *+[o][F]{b} \ar@{-}[r]^l \ar@{-}[d]^j  & *+[o][F]{} \\
*+[o][F]{}   \ar@{-}[r]_i  & *+[o][F]{} &}$$
In any case $(a\,b)\in \langle \rho_i,\rho_j,\rho_k\rangle\cap\langle \rho_i,\rho_j,\rho_l\rangle$ but $(a\,b)\notin \langle \rho_i,\rho_j\rangle$, a contradiction. 
\end{proof}

\begin{proposition}\label{graph}
If $\Gamma_i$ is intransitive for all $i\in\{0,\ldots, n-3\}$, then, up to a renumbering of the generators,  the permutation representation graph $\mathcal{G}$  of $\Gamma$
is one of the following graphs where $\mathcal{G}_{\{2,\ldots,n-3\}}$ is a tree.
$$(A)\xymatrix@-1.3pc{&&&&&&&\\*+[o][F]{}   \ar@{-}[r]^1& *+[o][F]{}   \ar@{-}[r]^0 & *+[o][F]{} \ar@{-}[r]^1& *+[o][F]{}\ar@{-}[r] ^2& *+[o][F]{}\ar@{.}[r]\ar@{.}[d]\ar@{.}[u]& *+[o][F]{}\ar@{-}[r]^i\ar@{.}[d]\ar@{.}[u]& *+[o][F]{}\ar@{.}[r] \ar@{.}[d]\ar@{.}[u]&*+[o][F]{}\ar@{.}[d]\ar@{.}[u]\\
&&&&&&&} \quad  (B) \xymatrix@-1.3pc{&&*+[o][F]{}  &&&&&\\*+[o][F]{}   \ar@{-}[r]^1& *+[o][F]{}   \ar@{-}[r]^0 & *+[o][F]{} \ar@{-}[u]_1\ar@{-}[r]_2 & *+[o][F]{}\ar@{.}[r]\ar@{.}[d]\ar@{.}[u]& *+[o][F]{}\ar@{-}[r]^i\ar@{.}[d]\ar@{.}[u]& *+[o][F]{}\ar@{.}[r] \ar@{.}[d]\ar@{.}[u]&*+[o][F]{}\ar@{.}[d]\ar@{.}[u]\\
&&&&&&&} (C) \xymatrix@-1.3pc{ *+[o][F]{}   \ar@{-}[r]^0\ar@{-}[d]_1  & *+[o][F]{} \ar@{-}[d]^1&&&&&& \\
*+[o][F]{}   \ar@{-}[r]_0  & *+[o][F]{} \ar@{-}[r]_2 & *+[o][F]{}\ar@{.}[r]\ar@{.}[d]\ar@{.}[u]& *+[o][F]{}\ar@{-}[r]^i\ar@{.}[d]\ar@{.}[u]& *+[o][F]{}\ar@{.}[r] \ar@{.}[d]\ar@{.}[u]&*+[o][F]{}\ar@{.}[d]\ar@{.}[u]\\
&&&&&&&} $$
\end{proposition}
\begin{proof}
This is a consequence of Lemmas~\ref{double}, \ref{nosquare} and \ref{square}.
\end{proof}

\section{Proof of Theorem~\ref{main}} \label{proof}

Let $n\geq 9$ and $\Gamma:= (S_n, \{\rho_0,\ldots, \rho_{n-3}\})$ be a C-group of rank $n-2$.
We observe that for $n=8$ there is a C-group having one maximal parabolic subgroup that is transitive, generated by the following set of involutions.$$\{(1\, 2),\,(1\, 2)(3\, 4),\,(1\, 2)(7\, 8),\,(1\, 2)(5\, 6),\,(1\, 3)(6\, 8),\,(1\, 8)(3\, 6)\}$$ This C-group also yields a hypertope. In this case the permutation representation graph is different from those given in Theorem~\ref{main}. 

Now suppose that $n\geq 9$. By Lemmas~\ref{pri<=n-4}, \ref{neverAnM12AGL} and \ref{imp}, all $\Gamma_i$'s are intransitive.  Thus by Proposition~\ref{graph} the permutation representation  graph of $\Gamma$ is one of the three possibilities given in Theorem~\ref{main}. It remains to prove that any group generated by involutions having one of these three graphs as permutation representation graph gives a C-group of rank $n-2$ and also a regular hypertope of rank $n-2$ for the symmetric group $S_n$. By Proposition~\ref{FTcgroup} we just need to prove that the groups satisfy the intersection property (and hence are C-groups) and that  the corresponding coset geometries are flag-transitive. 

We first focus on the intersection property.
In order to prove it, we need a slightly different result than Proposition 2E16 of~\cite{arp}.

\begin{proposition}\label{2E16}
Let $G$ be a group generated by $r$ involutions $g_0, \ldots, g_{r-1}$.
Suppose that every maximal parabolic subgroup $G_i$ is a C-group. Then $G$ is a C-group if and only if $G_i \cap G_j = G_{i,j}$ for all $0\leq i,j \leq r-1$.
\end{proposition}
\begin{proof}
Obviously, if $G_i \cap G_j \neq G_{i,j}$ for some $i,j$, then $G$ is not a C-group.
As pointed out in the proof of~\cite[Proposition 2E16]{arp}, to prove the intersection property, it suffices to show that
\[
G_K :=  \langle g_k| k \not\in K \rangle = \bigcap\{G_j | j \in K\}
\]
This follows immediately from the hypothesis that $G_i \cap G_j = G_{i,j}$ for all $0\leq i,j \leq r-1$.
\end{proof}
\begin{lemma}\label{Cgroup1}
The permutation representation graph (A) gives a rank $n-2$ C-group isomorphic to $S_n$.
\end{lemma}
\begin{proof}
By Proposition~\ref{2E16}, and using induction on $n$, it is sufficient to prove that  $\Gamma_i\cap \Gamma_j=\Gamma_{i,j}$ for every $i,j\in I$. Indeed, all $\Gamma_i$'s will either be C-groups of rank $n-3$ for $S_{n-1}$ or direct products of two smaller groups that are obviously C-groups.

$\Gamma_{0,1}=\Gamma_0\cap\Gamma_1$: we have $\Gamma_{0,1}\cong S_{n-3}$, thus $\Gamma_{0,1}$ is a maximal subgroup of $\Gamma_1\cong  2\times S_{n-3}$;

$\Gamma_{0,2}=\Gamma_0\cap\Gamma_2$: we have that $\Gamma_{0,2}\cong 2\times S_{n-4}$, and $\Gamma_2\cong  D_8\times S_{n-4}$. Suppose $\Gamma_0 \cap \Gamma_2 \neq \Gamma_{0,2}$. Then $\Gamma_0\cap \Gamma_2$ must be a proper subgroup of $\Gamma_0$ and a proper subgroup of $\Gamma_2$ containing $\Gamma_{0,2}$. The involution $\rho_1$ corresponds to a fixed-point-free reflection of a square in $D_8$. The only possibility is then to have another reflection or a central symmetry in $\Gamma_0\cap \Gamma_2$. But then $\Gamma_0\cap \Gamma_2$ is transitive on the four vertices of the square which is impossible as $\Gamma_0$ is not. Hence $\Gamma_0 \cap \Gamma_2 = \Gamma_{0,2}$
 
$\Gamma_{0,l}=\Gamma_0\cap\Gamma_l$ for $l\geq 3$: we have that $\Gamma_{0,l}\cong 2\times S_{n_1-2}\times  S_{n_2}$ with $n_1+n_2=n$, thus $\Gamma_{0,l}$ is a maximal subgroup of $\Gamma_0\cong  2\times S_{n-2}$;

$\Gamma_{1,2}=\Gamma_1\cap\Gamma_2$:  we have $\Gamma_{1,2}\cong 2\times S_{n-4}$, thus $\Gamma_{1,2}$  is a maximal subgroup of $\Gamma_1\cong  2\times S_{n-3}$;

$\Gamma_{1,l}=\Gamma_1\cap\Gamma_l$ for $l\geq 3$:  we have that $\Gamma_{1,l}\cong 2\times S_{n_1-3}\times  S_{n_2}$ with $n_1+n_2=n$, thus $\Gamma_{1,l}$ is a maximal subgroup of $\Gamma_1\cong  2\times S_{n-3}$;

$\Gamma_{2,l}=\Gamma_2\cap\Gamma_l$ for $l\geq 3$: we have that $\Gamma_{2,l}\cong D_8\times S_{n_1-4}\times  S_{n_2}$ with $n_1+n_2=n$, thus $\Gamma_{2,l}$ is a maximal subgroup of $\Gamma_2\cong  D_8\times S_{n-4}$;

$\Gamma_{l,k}=\Gamma_l\cap\Gamma_k$ for $l,k\geq 3$: we have that $\Gamma_{l,k}\cong S_{n_1}\times  S_{n_2}\times S_{n_3}\times  S_{n_4}$ with $n_1+n_2+n_3+n_4=n$, thus $\Gamma_{l,k}$ is a maximal subgroup of $\Gamma_l\cong  S_{n_1}\times  S_{n_2}\times S_{n_3+n_4}$.

\end{proof}

\begin{lemma}\label{Cgroup2}
The permutation representation graph (B) gives a rank $n-2$ C-group isomorphic to $S_n$.
\end{lemma}
\begin{proof}
By Proposition~\ref{2E16}, and using induction on $n$ as in the previous lemma, it is sufficient to prove that  $\Gamma_i\cap \Gamma_j=\Gamma_{i,j}$ for every $i,j\in I$. 

To prove the equalities $\Gamma_{0,2}=\Gamma_0\cap\Gamma_2$, $\Gamma_{1,2}=\Gamma_1\cap\Gamma_2$, $\Gamma_{0,l}=\Gamma_0\cap\Gamma_l$ for $l\geq 3$, $\Gamma_{2,l}=\Gamma_2\cap\Gamma_l$ for $l\geq 3$, and $\Gamma_{l,k}=\Gamma_l\cap\Gamma_k$ for $l,k\geq 3$,  we can use the same argument used  in Lemma~\ref{Cgroup1}.  Hence to prove that $\Gamma$ is a C-group only the following two equalities are needed.

$\Gamma_{0,1}=\Gamma_0\cap\Gamma_1$: we have $\Gamma_{0,1}\cong S_{n-3}$, thus is $\Gamma_{0,1}$ a maximal subgroup of $\Gamma_1\cong  S_{n-2}$;

$\Gamma_{1,l}=\Gamma_1\cap\Gamma_l$ for $l\geq 3$: we have that $\Gamma_{1,l}\cong S_{n_1-2}\times  S_{n_2}$ with $n_1+n_2=n$, thus $\Gamma_{1,l}$ is a maximal subgroup of $\Gamma_1\cong   S_{n-2}$;

\end{proof}
\begin{lemma}\label{Cgroup3}
The permutation representation graph (C) gives a rank $n-2$ C-group isomorphic to $S_n$.
\end{lemma}
\begin{proof}
By Proposition~\ref{2E16}, and using induction on $n$ as in the previous lemmas, it is sufficient to prove that  $\Gamma_i\cap \Gamma_j=\Gamma_{i,j}$ for every $i,j\in I$. 

As $\mathcal{G}_{0,1}$ is a tree with $n-3$ vertices, $\Gamma_{0,1}$ is a C-group isomorphic to $S_{n-3}$. In addition we have the following equalities:

 $\Gamma_{0,1}=\Gamma_0\cap\Gamma_1$: we have $\Gamma_{0,1}\cong S_{n-3}$ and $\Gamma_0\cong  2\times S_{n-2}$. Obviously, $\Gamma_0\cap\Gamma_1$ has three fixed points, hence the equality follows.

$\Gamma_{0,2}=\Gamma_0\cap\Gamma_2$: we have $\Gamma_{0,2}\cong 2\times S_{n-4}$, thus $\Gamma_{0,2}$  is a maximal subgroup of $\Gamma_2\cong  2^2\times S_{n-4}$;
 
 $\Gamma_{0,l}=\Gamma_0\cap\Gamma_l$ for $l\geq 3$: we have that $\Gamma_{0,l}\cong 2\times S_{n_1-2}\times  S_{n_2}$ with $n_1+n_2=n$, thus $\Gamma_{0,l}$ is a maximal subgroup of $\Gamma_0\cong  2\times S_{n-2}$;

 $\Gamma_{1,2}=\Gamma_1\cap\Gamma_2$: we have that $\Gamma_{1,2}\cong 2\times S_{n-4}$, thus $\Gamma_{1,2}$ is a maximal subgroup of $\Gamma_2\cong  2^2\times S_{n-4}$; 

$\Gamma_{1,l}=\Gamma_1\cap\Gamma_l$ for $l\geq 3$: we have that $\Gamma_{1,l}\cong 2\times S_{n_1-2}\times  S_{n_2}$ with $n_1+n_2=n$, thus $\Gamma_{1,l}$ is a maximal subgroup of $\Gamma_1\cong  2\times S_{n-2}$;

$\Gamma_{2,l}=\Gamma_2\cap\Gamma_l$ for $l\geq 3$: we have that $\Gamma_{2,l}\cong 2^2\times S_{n_1-4}\times  S_{n_2}$ with $n_1+n_2=n$, thus $\Gamma_{1,l}$ is a maximal subgroup of $\Gamma_2\cong  2^2\times S_{n-2}$;

$\Gamma_{l,k}=\Gamma_l\cap\Gamma_k$ for $l,k\geq 3$: we have that $\Gamma_{l,k}\cong S_{n_1}\times  S_{n_2}\times S_{n_3}\times  S_{n_4}$ with $n_1+n_2+n_3+n_4=n$, thus $\Gamma_{l,k}$ is a maximal subgroup of $\Gamma_l\cong  S_{n_1}\times  S_{n_2}\times S_{n_3+n_4}$.

Hence we have proved that $\Gamma_i\cap \Gamma_j=\Gamma_{i,j}$ for every $i,j\in I$ which is sufficient to show that $\Gamma$ is a C-group.

\end{proof}

The following corollary gives the Coxeter diagrams of the C-groups of Theorem~\ref{main} as well as presentations for these groups.

\begin{corollary}\label{coro6.5}
Let $n\geq 9$ and  $\Gamma:=(S_n,\{\rho_0,\ldots,\rho_{n-3}\})$ be  a C-group of rank $n-2$ with one of the three possible permutation representations given in Theorem~\ref{main}.
\begin{enumerate}
\item The Coxeter diagram of $\langle \rho_0,\rho_1,\rho_2,\rho_3\rangle$ is one of the following, accordantly to its permutation representation graph.
$$ (A) \qquad
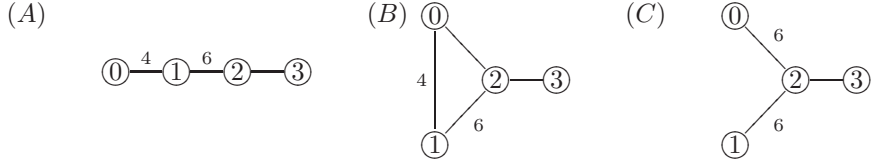
\xymatrix@-1pc{ &&&\\
*+[o][F]{0}   \ar@{-}[r]^4 & *+[o][F]{1}  \ar@{-}[r]^6  & *+[o][F]{2} \ar@{-}[r]& *+[o][F]{3}\\
&&&}
\qquad (B) \xymatrix@-1pc{  *+[o][F]{0}   \ar@{-}[dd]_4\ar@{-}[dr] & &\\
& *+[o][F]{2}  \ar@{-}[ld]^6\ar@{-}[r]& *+[o][F]{3}\\
 *+[o][F]{1} && } \qquad (C) \qquad\xymatrix@-1pc{  *+[o][F]{0}   \ar@{-}[dr]^6 & &\\
& *+[o][F]{2}  \ar@{-}[ld]^6\ar@{-}[r]& *+[o][F]{3}\\
 *+[o][F]{1} && }$$
 \item $\rho_0, \,\rho_1$ and $\rho_2$ commute with $\rho_i$ for $i\geq 4$;
\item The Coxeter diagram of $\langle \rho_3,\ldots,\rho_{n-3}\rangle$ is the line graph of $\mathcal{G}_{3,\ldots,n-3\}}$ that is an IMG graph;
\item If  $\{\rho_i\rho_j\rho_k\}$, for $i,j,k\geq 3$, is a triangle of the Coxeter diagram then $(\rho_i\rho_j\rho_i\rho_k)^2=1$;
\item The hypertopes with permutation representation $(A)$, $(B)$ and $(C)$ satisfy the following relations, respectively. 
$$(A)\quad [(\rho_1\rho_2)^3\rho_0]^3=[\rho_3(\rho_1\rho_2)^3]^2=(\rho_0\rho_1\rho_2\rho_1)^3=1$$
$$ (B)\quad  [\rho_0(\rho_1\rho_2)^3]^3=(\rho_1\rho_0\rho_1\rho_2)^2=[(\rho_1rho_2)^3\rho_3]^2=1$$
$$ (C)\quad [(\rho_0\rho_2)^3(\rho_1\rho_2)^3\rho_1]^3=[(\rho_0\rho_2)^3(\rho_1\rho_2)^3]^3=[(\rho_1\rho_2)^3\rho_3]^2=[(\rho_0\rho_2)^3\rho_3]^2=1$$
\end{enumerate}
Moreover the relations given by the Coxeter diagram (characterized by (1), (2) and (3)), plus the relations given in (4) and (5)  are sufficient to define the automorphism group $\Gamma$.
\end{corollary}
\begin{proof}
If $\Gamma$ is a C-group of rank $n-2$ for $S_n$ with $n\geq 9$ then by  Theorem~\ref{main} and Proposition~\ref{preS} all items of this Corollary can be easily verified. In what follows we prove that these relations are sufficient  to characterise the groups of these C-groups.

We proceed by  induction over the number of nonlabelled triangles of the Coxeter diagram.
Starting from a C-group of rank $n-2$ for $S_n$ having $t\geq 1$ nonlabelled triangles on its Coxeter diagram we can reduce the number of triangles using the construction given in Proposition~\ref{preS}. Note that in any of the three permutation representations $(A)$, $(B)$ and $(C)$ it is possible to apply the construction used in Proposition~\ref{preS}. Indeed, if $t>0$, in any of the three cases there are two vertices of degree one and a path between them not containing the first four generators. Thus only the cases without nonlabelled triangles ($t=0$) need to be analysed. In what follows let $\alpha_0,\ldots, \alpha_{n-2}$ be the generators of the string C-group $[3^{n-2}]$. 

Let us first consider the C-group with permutation representation $(A)$ and $t=0$, that is, the string C-group of rank $n-2$ for the symmetric group of type $\{4,6,3^{n-5}\}$. We can derive a presentation of the group of this C-group  from the finite Coxeter group $[3^n]$. Let
$$\rho_0=\alpha_1,\;\rho_1=\alpha_0\alpha_2\mbox{ and  }\rho_i=\alpha_{i+1}\mbox{ for }i\in\{2,\ldots, n-3\}. $$
where $\alpha_{0},\ldots,\alpha_{n-2}$ are the standard generators of $[3^n]$.
We have, 
$$\alpha_0=(\rho_1\rho_2)^3\mbox{ and }\alpha_2=(\rho_1\rho_2)^3\rho_1.$$
 Now from the relations of the Coxeter group $[3^n]$ we derive, apart from the relations giving the type, two other relations: $((\rho_1\rho_2)^3\rho_0)^3=1$, $((\rho_1\rho_2)^3\rho_3)^2=1$ and $(\rho_0\rho_1\rho_2\rho_1)^3=1$. This proves that the relations given is this corollary are sufficient to give a presentation of the group of this C-group. 
 
 Now consider the C-group with permutation representation $(B)$ and $t=0$. In this case let,
 $$\rho_0=\alpha_2,\, \rho_1=(\alpha_0,\alpha_2\alpha_1)^2,\, \rho_i=\alpha_{i+1}\mbox{ for }i\in\{2,\ldots, n-3\}.$$
 We have, 
$$\alpha_0=\rho_1\rho_0\rho_1\mbox{ and }\alpha_1=(\rho_1\rho_2)^3.$$
 Now from the relations of the Coxeter group $[3^n]$ we derive, apart from the relations giving the type, only two extra relations:  $[\rho_0(\rho_1\rho_2)^3]^3=(\rho_1\rho_0\rho_1\rho_2)^2=[(\rho_1\rho_2)^3\rho_3]^2=1$. This proves that the relations given is this corollary are sufficient to give a presentation of the group of this C-group.
 
For the C-group with permutation representation (C) and $t=0$ we again derive a presentation from the group of the C-group without triangles in its Coxeter diagram from  $[3^n]$ changing the generating set as follows: 
$$\rho_0=(\alpha_0\alpha_2\alpha_1)^2,\;\rho_1=\alpha_0\alpha_2\mbox{ and }\rho_i=\alpha_{i+1}\mbox{ for }i\in\{2,\ldots, n-3\}$$
On the other hand we have, 
$$\alpha_0=(\rho_1\rho_2)^3, \;\alpha_1=(\rho_0\rho_2)^3\mbox{ and }\alpha_2=(\rho_1\rho_2)^3\rho_1.$$
Now from the relations of the Coxeter group $[3^n]$ we derive, apart from the relations giving the type, the two relations given in this corollary. 
\end{proof}
 
Observe that (1), (2) and (3) of Corollary~\ref{coro6.5} permit to say that the number of C-groups of rank $n-2$ up to isomorphism and duality is divisible by 3. For $S_n$ with $n\geq 9$, it corresponds to three times the number of leaves in the permutation representation graphs of the inductively minimal geometries of $S_{n-3}$.

It now remains to prove that the coset geometries obtained from the permutation representation graphs (A), (B) and (C) are all flag-transitive in order to show that these C-groups are all giving regular hypertopes.
\begin{theorem}
The coset geometries obtained from the permutation representation graphs (A), (B) and (C) are flag-transitive.
\end{theorem}
\begin{proof}
Let $\Gamma$ be a coset geometry of rank $n-2$ obtained from one of the permutation representation graphs.
Recall that for each such geometry, the maximal parabolic subgroups $\Gamma_i$ are intransitive subgroups on $n$ points.
Choose one leaf on the right hand side of the permutation representation graph of the C-groups we obtained,  and suppose its edge-label is $n-3$. Then $\Gamma_{n-3}$ is isomorphic to $S_{n-1}$ by induction and $\Gamma(\Gamma_{n-3},\{\Gamma_0 \cap \Gamma_{n-3}, \ldots,$ $\Gamma_{n-4} \cap \Gamma_{n-3}\} )$ is a flag-transitive geometry.
Using Theorem~\ref{FTlee2}, we just need to check that $\Gamma_{(ij)}(S_n,$ $\{\Gamma_i,\Gamma_j,\Gamma_{n-3}\})$ is flag-transitive for all $0<i,j < n-4$ in order to prove that $\Gamma$ is flag-transitive.
Observe that by Lemmas~\ref{Cgroup1},~\ref{Cgroup2} and~\ref{Cgroup3}, we have $\Gamma_i \cap \Gamma_j = \Gamma_{i,j}$.
Let $p$ be the point fixed by $\Gamma_{n-3}$.
We have that $\Gamma_{n-3}(\Gamma_i\cap \Gamma_j) = \Gamma_{n-3}\Gamma_{i,j}$ and since $\Gamma_{n-3} \cong S_{n-1}$, this is the set of all the elements of $S_n$ that map $p$ to any point of the orbits $p^{\Gamma_{i,j}}$.
Now, assume by way of contradiction that $H:= \Gamma_{n-3}\Gamma_i \cap \Gamma_{n-3}\Gamma_j > \Gamma_{n-3}(\Gamma_i\cap \Gamma_j)$. Then there exists $\alpha\in H$ such that $p\alpha =: q \not\in p^{\Gamma_{i,j}}$. This point must then be either in $p^{\Gamma_{i}}\setminus p^{\Gamma_{j}}$ or in  $p^{\Gamma_{j}}\setminus p^{\Gamma_{i}}$. Suppose without loss of generality that $q \in p^{\Gamma_{i}}\setminus p^{\Gamma_{j}}$. 
As $\alpha\in H$,  $\alpha = \gamma_{n-3}\gamma_i = \nu_{n-3} \nu_j$ for some $\gamma_{n-3},\nu_{n-3}\in \Gamma_{n-3}$, $\gamma_i\in\Gamma_i$ and $\nu_j\in\Gamma_j$. But since  $p\alpha=p(\nu_{n-3}\nu_j)=p \nu_j$  we cannot have $p\alpha = q$, a contradiction. 
Hence  $\Gamma_{(ij)}(S_n,$ $\{\Gamma_i,\Gamma_j,\Gamma_{n-3}\})$ is flag-transitive for all $0 \leq i,j \leq n-4$ and by Theorem~\ref{FTlee2}, $\Gamma$ is flag-transitive.
\end{proof}
\section{Acknowledgements}
This research was supported by a Marsden grant (UOA1218) of the Royal Society of New Zealand and  by the Portuguese funds through the CIDMA - Center for Research and Development in Mathematics and Applications, and the Portuguese Foundation for Science and Technology (FCT- Fundação para a Ciência e a Tecnologia), within project PEst-OE/MAT/UI4106/2014.

\bibliographystyle{plain}

\end{document}